%% file: isometric-rigidity-rev.tex
\documentclass[english,12pt]{smfart}
 \pdfoutput=1
 \usepackage{cmap}
  \usepackage{smfthm,smfenum}
  \usepackage{amsmath,amsxtra,amscd,amssymb,latexsym,stmaryrd}
  \usepackage{mathrsfs,dsfont}
\usepackage{graphicx}
\usepackage{color}

\newcommand{\R}{\mathbb{R}}
\newcommand{\Ha}{\mathcal{H}}
\renewcommand{\S}{\mathbb{S}}
\newcommand{\N}{\mathbb{N}}

\newcommand{\Reg}{\operatorname{Reg}}

\title[Geometric Wasserstein spaces]{A geometric study of Wasserstein spaces: isometric rigidity
  in negative curvature}

\author{J\'er\^ome Bertrand}
\address{Institut de Math\'ematiques\\ Universit\'e Paul Sabatier \\ 118 route de
  Narbonne \\ F31062 Cedex 9 Toulouse\\ France}
\email{bertrand@math.univ-toulouse.fr}

\author{Beno\^{\i}t R. Kloeckner}
\address{Universit\'e Paris-Est, Laboratoire d'Analyse et de 
Math\'ematiques Appliqu\'ees (UMR 8050), UPEM, UPEC, CNRS, F-94010, 
Cr\'eteil, France \quad\textit{and} \newline
Univ. Grenoble Alpes, IF, F-38000 Grenoble, France \newline
CNRS, IF, F-38000 Grenoble, France}
\email{benoit.kloeckner@u-pec.fr}

\theoremstyle{plain}

\newtheorem*{acknowledgements}{Acknowledgements}

\newcommand{\CAT}{\mathop{\mathrm{CAT}}\nolimits} 
\newcommand{\isom}{\mathop{\mathrm{Isom}}\nolimits}
\newcommand{\wass}{\mathop{\mathscr{W}_2}\nolimits}
\newcommand{\radon}{\mathop{\mathscr{R}}\nolimits}
\newcommand{\dw}{\mathop{\mathrm{W}}\nolimits}

\newcommand{\supp}{\mathop{\mathrm{supp}}\nolimits}
\newcommand{\perpend}{\mathop{\perp}}

\newcommand{\RH}{\mathbb{R}\mathrm{H}}

\begin{document}
\begin{abstract}
Given a metric space $X$, one defines its Wasserstein space $\wass(X)$
as a set of sufficiently decaying probability measures on $X$ endowed with a metric
defined from optimal transportation. In this article, we continue the geometric study
of $\wass(X)$ when $X$ is a simply connected, nonpositively curved metric spaces
by considering its isometry group. When $X$ is Euclidean, the second named author
proved that this isometry group is larger than the isometry group of $X$. In contrast,
we prove here a rigidity result: when $X$ is negatively curved, any isometry
of $\wass(X)$ comes from an isometry of $X$.
\end{abstract}

\thanks{This work was supported by the Agence Nationale de la Recherche, grant ANR-11-JS01-0011.}

\maketitle

\section{Introduction}

This article is part of a series where the Wasserstein space of a metric
space is studied from an intrinsic, geometric point of view. Given a
Polish metric space $X$, the set of its Borel probability measures
of finite second moment can be endowed \textit{via} optimal transport
with a natural distance; the resulting metric space is
the \emph{Wasserstein space} $\wass(X)$ of $X$. We shall only give a minimal amount of
background and we shall not discuss previous works, so as to avoid redundancy with the
previous articles in the series. One can think of these Wasserstein spaces as geometric 
measure theory analogues of $L^p$ spaces (here with $p=2$), but they recall 
much more of the geometry of $X$. Our goal is to understand precisely how the geometric
properties of $X$ and $\wass(X)$ are related.

\subsection{Main results}

Our setting here is at the intersection of the previous papers
\cite{Kloeckner} and \cite{Bertrand-Kloeckner}: as in the later we assume
$X$ is a Hadamard space, by which we mean a 
globally $\CAT(0)$, complete, locally compact space, and as in the former we consider
the isometry group of $\wass(X)$. Note that a weaker result concerning the isometric rigidity was included in 
a previous preprint version of \cite{Bertrand-Kloeckner}, which has been divided
due to length issues after a remark from a referee.

In \cite{Kloeckner}
it was proved that the isometry group of $\wass(\mathbb{R}^n)$ is strictly larger than 
the isometry group
of $\mathbb{R}^n$ itself, the case $n=1$ being the most striking: some isometries of 
$\wass(\mathbb{R})$
are exotic in the sense that they do not preserve the shape of measures. This
property seems pretty uncommon, and our main result is the following.
\begin{theo}\label{theo:intro:isometry}
If $X$ is a negatively curved geodesically complete Ha\-damard space then $\wass(X)$ is isometrically 
rigid in the sense that its only isometries are those induced by the isometries of $X$ itself.
\end{theo}
By ``Negatively curved'' we mean that $X$ is not a line and 
the $\CAT(0)$ inequality is strict except for triangles all of whose vertices are aligned,
see Section \ref{had} for details. 

As stated, our proof of Theorem \ref{theo:intro:isometry}
depends on strong results of Lytchak and Nagano \cite{LN}; but we shall treat many cases (manifolds, trees, some more general polyhedral complexes) without resorting
to \cite{LN}.

In the process of proving Theorem \ref{theo:intro:isometry}, we also get the following result that seems interesting by itself.
\begin{theo}\label{theo:side}
Let $Y,Z$ be geodesic and geodesically complete Polish spaces and assume that $Y$ is locally compact.
Then $\wass(Y)$ is isometric to $\wass(Z)$ if and only if $Y$ is isometric to $Z$.
\end{theo}
\begin{rema}
The statement of this result in the published version of the present paper is more general in that it does not ask $Z$ to be geodesically complete. However, we are only able to prove the above more restricted statement. We thank Zhengchao Wan for bringing this issue to our attention.
\end{rema}
It is quite surprising that whether a fully general version of this result holds is an open
question.

Note that the proof of Theorem \ref{theo:intro:isometry} involves the 
inversion of some kind of
Radon transform, that seems to be new (for trees it is in particular different from the horocycle
Radon transform) and could be of interest for other problems.

\subsection{Preliminaries and basic notions}

In this preliminary section, for the sake of self-containment,
we briefly recall well-known general facts on Hadamard and
Wasserstein spaces. One can refer
to \cite{Ballmann,BH} and
\cite{Villani,Villani1} for proofs, further details and much more.

\subsubsection{Wasserstein space}

Given a Polish (\textit{i.e.} complete and separable metric) space $X$,
one defines its (quadratic) Wasserstein space $\wass(X)$ as the set
of Borel probability measures $\mu$ on $X$ that satisfy
\[\int_X d^2(x_0,x) \mu(dx)<+\infty\]
for some (hence all) point $x_0\in X$, equipped by the distance $\dw$ defined by:
\begin{equation}
\dw^2(\mu_0,\mu_1)=\inf \int_{X\times X} d^2(x,y) \Pi(dxdy)
\label{eq:Wasserstein}
\end{equation}
where the infimum is taken over all measures $\Pi$ on $X\times X$
whose marginals are $\mu_0$ and $\mu_1$. Such a measure is called a transport plan and
is said to be optimal if it achieves the infimum.
In this setting, optimal plans always exist and $\dw$ turns
$\wass(X)$ into a metric space.

\subsubsection{Hadamard spaces}\label{had}

Given any three points $x,y,z$ in a geodesic Polish metric space $X$,
there is up to congruence a unique comparison triangle $x',y',z'$ in $\mathbb{R}^2$,
that is a triangle that satisfies $d(x,y)=d(x',y')$, $d(y,z)=d(y',z')$,
and $d(z,x)=d(z',x')$.

One says that $X$ has (globally) \emph{non-positive curvature} (in the sense of Alexandrov),
or is $\CAT(0)$,
if for all $x,y,z$ the distances between two points on sides of this triangle
is lesser than or equal to the distance between the corresponding points in the comparison
triangle, see figure \ref{fig:comparison}.
Note that some authors call this property ``globally $\CAT(0)$''.

Equivalently, $X$ is $\CAT(0)$ if for any triangle $x,y,z$, any geodesic
$\gamma$ such that $\gamma_0=x$ and $\gamma_1=y$, and any $t\in[0,1]$,
the following inequality holds:
\begin{equation}
d^2(y,\gamma_t)\leqslant (1-t)d^2(y,\gamma_0)+td^2(y,\gamma_1)-t(1-t)\ell(\gamma)^2
\label{eq:cat0}
\end{equation}
where $\ell(\gamma)$ denotes the length of $\gamma$ (which is equal to $d(x,z)$).

\begin{figure}[ht]\begin{center}
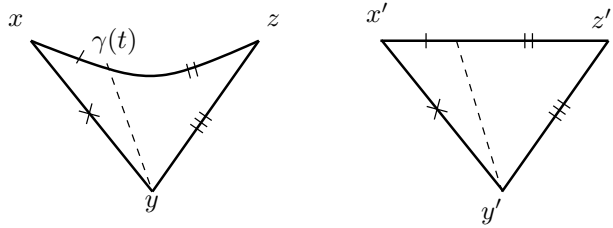
\caption{The $\CAT(0)$ inequality: the dashed segment is shorter in the triangle $xyz$
than in the comparison triangle on the right.}\label{fig:comparison}
\end{center}\end{figure}

We shall say that $X$ is a Hadamard space if $X$ is $\CAT(0)$, complete and locally compact.
By the generalized Cartan-Hadamard theorem, this implies that $X$ is simply connected.
Our goal is now, as said above, to study the isometry group of $\wass(X)$.

\subsection{Strategy of proof}

We have a natural morphism 
\begin{eqnarray*}
\# : \isom X  &\to& \isom\wass(X) \\ 
     \varphi  &\mapsto& \varphi_\#
\end{eqnarray*}
where $\varphi_\#$ is the push-forward of measures:
\[\varphi_\#\mu(A):=\mu(\varphi^{-1}(A)).\]
An isometry $\Phi$
of $\wass(X)$ is said to be \emph{trivial} if it is in the image
of $\#$, to \emph{preserve shape} if for all $\mu\in\wass(X)$ there is an isometry
$\varphi^\mu$ of $X$ such that $\Phi(\mu)=\varphi^\mu_\#\mu$, and to be \emph{exotic}
otherwise. When all isometries of $\wass(X)$ are trivial, that is when
$\#$ is an isomorphism, we say that $\wass(X)$ is \emph{isometrically rigid}.

When $X$ is Euclidean, $\wass(X)$ is not isometrically rigid as proved in \cite{Kloeckner}.

Our main result, Theorem \ref{theo:intro:isometry} can now be phrased as follows:
if $X$ is a geodesically complete, negatively curved Hadamard space,
then $\wass(X)$ is isometrically rigid.

We first prove this assuming the set of regular points (namely those 
whose tangent cone is isometric to a Euclidean space plus an extra 
assumption if the dimension is one, see Section \ref{blurb} for a 
precise definition) is a dense subset of $X$. The density of regular points 
holds in many examples such as manifolds, simplicial trees, polyhedral complexes whose 
faces are endowed with hyperbolic metrics such as $I_{p,q}$ buildings.
Building on results of Lytchak and Nagano \cite{LN}, we shall prove that the density
of regular points holds true for \emph{any} geodesically complete, negatively curved Hadamard space,
completing the proof of Theorem \ref{theo:intro:isometry}.

To  prove isometric rigidity, we show first that an isometry 
must map Dirac masses 
to Dirac masses. The method is similar to that used in 
the Euclidean case, and is valid for geodesically complete locally compact spaces
without curvature assumption.
It already yields Theorem \ref{theo:side}, answering positively
in a wide setting the following question: if
$\wass(Y)$ and $\wass(Z)$ are isometric, can one conclude that $Y$ and $Z$ are isometric?

The second step is to prove that an isometry that fixes all Dirac masses must
map a measure supported on a geodesic to a measure supported on the same geodesic.
Here we use that $X$ has negative curvature: this property is known to be false
in the Euclidean case.

Then as in \cite{Kloeckner} we are reduced to prove the injectivity of a specific Radon 
transform. The point is that this injectivity is known for $\RH^n$ and
is easy to prove for trees, but seems not to be known for other spaces. We give a
simple argument for manifolds, then extend it to spaces with a dense set of regular 
points. Finally, we prove that all 
geodesically complete, negatively curved Hadamard spaces have a dense set of regular points.

\section{Reduction to a Radon transform}\label{sec:isometries1}

In this Section we reduce Theorem \ref{theo:intro:isometry} to the injectivity
of a kind of Radon transform, and also prove Theorem \ref{theo:side}.

\subsection{Characterization of Dirac masses}

The characterization of Dirac masses follows from two lemmas, and can be carried
out without any assumption on the curvature. On the contrary, the assumption of {\it gedesic completeness}, 
namely that any geodesic can be extended to a (minimizing) geodesic defined on $\R$, is crucial. Note that when we ask a space
to be geodesically complete, we of course imply that it is geodesic. Note that
all geodesics are assumed to be parameterized with constant speed and to be globally minimizing.

\begin{lemm}\label{lem:extension}
Let $Y$ be a locally compact, geodesically complete Polish space.
Any geode\-sic segment in $\wass(Y)$
issued from a Dirac mass can be extended to a complete geodesic ray.
\end{lemm}

\begin{proof}
Let us recall the measurable selection theorem 
(see for example \cite{Dellacherie}, Corollary of Theorem 17):
any surjective measurable map between Polish spaces admits a
measurable right inverse provided its fibers are compact. Consider
the restriction map
\[p^T : \mathscr{R}(Y) \to \mathscr{G}^{0,T}(Y)\]
where $\mathscr{G}^{0,T}(Y)$ is the set of geodesic segments parameterized on $[0,T]$
and $\mathscr{R}(Y)$ is the set of complete rays (geodesics are assumed here to be minimizing
and to have constant, non-necessarily unitary, possibly zero speed and all sets of geodesics
are endowed with the topology of uniform convergence on compact sets).

The fiber of a geodesic segment $\gamma$
is closed in the set $\mathscr{R}_{\gamma_0, s}(Y)$
of geodesic rays of speed $s=s(\gamma)$ starting at $\gamma_0$.
This set is compact by Arzela-Ascoli theorem (equicontinuity
follows from the speed being fixed, while the pointwise relative compactness
is a consequence of $Y$ being locally compact and geodesic, hence proper).

Moreover the geodesic completeness implies the surjectivity 
of $p^T$. There is therefore a measurable right inverse
$q$ of $p^T$: it maps a geodesic segment $\gamma$ to a complete ray whose restriction
to $[0,T]$ is $\gamma$.

Let $(\mu_t)_{t\in[0,T]}$ be a geodesic segment of $\wass(Y)$
of speed $\bar{s}$, with $\mu_0=\delta_x$ for some point $x$.

For $t\in[0,T]$, let 
$e_t:\mathscr{G}^{0,T}(Y)\to Y$
be the evaluation map at time $t$. We know from the theory of optimal transport
that there is a measure $\mu'$ on $\mathscr{G}^{0,T}(Y)$ such that $\mu_t=(e_t)_\#\mu'$,
called the displacement interpolation of the geodesic segment.

Let $\mu=q_\#(\mu')$: it is a probability measure
on $\mathscr{R}(Y)$ whose restriction $p^T_\#(\mu)$ is the
displacement interpolation of $(\mu_t)$. For all $t>0$ we therefore denote
by $\mu_t$ the measure $(e_t)_\#\mu$ on $X$; for $t\le T$, we retrieve the original
mesure $\mu_t$.

It is easy to see that $(\mu_t)_{t\ge 0}$ now defines
a geodesic ray: since there is only one transport plan from a Dirac mass to
any fixed measure,
$\dw^2(\mu_0,\mu_t)= \int s(\gamma)^2 t^2\, \mu(d\gamma) = \bar{s}^2 t^2$.
Moreover the transport plan from $\mu_t$ to $\mu_{t'}$ deduced from $\mu$ gives
$\dw(\mu_t,\mu_{t'})\leqslant \bar{s}|t-t'|$. But the triangular inequality
applied to $\mu_0$, $\mu_t$ and $\mu_{t'}$ implies
$\dw(\mu_t,\mu_{t'})\geqslant \bar{s}|t-t'|$ and we are done.
\end{proof}

\begin{lemm}\label{lemm:no-ext}
If $Y$ is a geodesic Polish space, given $\mu_0\in\wass(Y)$ not a Dirac mass
and $y\in\supp\mu_0$,
the geodesic segment from $\mu_0$ to $\mu_1:=\delta_y$ cannot be geodesically
extended for any time $t>1$.
\end{lemm}

Note that at least in the branching case the assumption $y\in\supp\mu_0$ is 
needed.

\begin{proof}
Assume that there is a geodesic segment $(\mu_t)_{t\in [0,1+\varepsilon]}$
and let $\mu$ be one of its displacement interpolation. Since $y\in\supp\mu_0$, 
$\supp\mu$ contains a geodesic segment of $X$ that is at $y$ at times $0$ and
$1$, and must therefore be constant.

Since $\mu_0$ is not a Dirac mass, there is a $y'\neq y$ in $\supp\mu_0$.
Let $\gamma$ be a geodesic segment in $\supp\mu$ such that $\gamma_0=y'$.
Then $\gamma_1=y$ lies between $y'$ and $y'':=\gamma_{1+\varepsilon}$
on $\gamma$, and the transport plan $\Pi$ from $\mu_0$ to $\mu_{1+\varepsilon}$
defined by $\mu$ contains in its
support the couples $(y', y'')$ and $(y,y)$.

But then cyclical monotonicity shows that $\Pi$ is not optimal: it costs less
to move $y$ to $y''$ and $y'$ to $y$ by convexity of the cost
(see figure \ref{fig:aligned}).
The dynamical transport $\mu$ cannot be optimal either, a contradiction.
\end{proof}

\begin{figure}[htp]\begin{center}
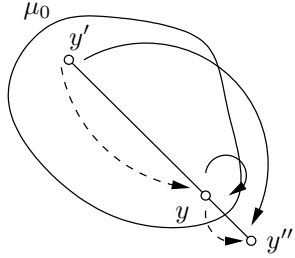
\caption{The transport shown with continuous arrows is less effective
than the transport given by the dashed arrows.}\label{fig:aligned}
\end{center}\end{figure}

We can now easily draw the consequences of these lemmas.

\begin{prop}
Let $Y,Z$ be geodesic Polish spaces and assume that $Y$ is geodesically complete and
locally compact.
Any isometry from $\wass(Y)$ to $\wass(Z)$ must map all Dirac masses to Dirac masses.
\end{prop}

Except in the next corollary, we shall use this result with $Y=Z=X$.
Note that we will need to require $X$ being geodesically complete in addition
to the Hadamard hypothesis.

\begin{proof}
Denote by $\varphi$ an isometry $\wass(Y)\to\wass(Z)$ and consider any
$x\in Y$.
If $\varphi(\delta_x)$ were not a Dirac mass, there would exist
a geodesic segment $(\mu_t)_{t\in[0,1]}$ from $\varphi(\delta_x)$ to a Dirac mass (at a point 
$y\in\supp  \varphi(\delta_x)$) that cannot be extended for times $t>1$.

But $\varphi^{-1}(\mu_t)$ gives a geodesic segment issued from $\delta_x$,
that can therefore be extended. This is a contradiction since $\varphi$ is an isometry.
\end{proof}

We can now prove Theorem \ref{theo:side}, which we recall:
let $Y,Z$ be geodesic, geodesically complete Polish spaces and assume that $Y$ is locally compact;
then $\wass(Y)$ is isometric to $\wass(Z)$ if and only if $Y$ is isometric to $Z$.

\begin{proof}[Proof of Theorem \ref{theo:side}]
Let $\varphi$ be an isometry $\wass(Y)\to \wass(Z)$. Then $\varphi$
maps Dirac masses to Dirac masses, and since the set of Dirac masses
of a space is canonically isometric to the space, $\varphi$ induces an isometric embedding $Y\to Z$. It suffices to prove that this isometric embedding is onto to deduce that $Y$ and $Z$ are isometric. Assume otherwise; then there exist $z_0\in Z$ such that $\mu_0:=\varphi^{-1}(\delta_{z_0})$ is not a Dirac mass. Pick any $y\in\supp\mu_{z_0}$ and let $z_1$ be such that $\varphi(\delta_y)=\delta_{z_1}$. Since $Z$ is geodesically complete, there is a complete geodesic $(z_t)_{t\in\mathbb{R}}$ in $Z$ passing through $z_0$ and $z_1$. Then $(\varphi^{-1}(\delta_{z_t}))_{t\in\mathbb{R}}$ is a geodesic in $\wass(Y)$ passing through $\mu_0$ and $\delta_y$, in contradiction with Lemma \ref{lemm:no-ext}. Therefore, if $\wass(Y)$ and $\wass(Z)$ are isometric, then so are $Y$ and $Z$.
The converse implication is obvious.
\end{proof}

Note that we do not know whether this result holds for general
metric spaces.
Also, we do not know if there is a space $Y$ and an isometry
of $\wass(Y)$ that maps some Dirac mass to a measure that is not a Dirac mass.

\subsection{Measures supported on a geodesic}

The characterization of measures supported on a geodesic relies on the following argument:
when dilated from a point of the geodesic, such a measure has Euclidean expansion.

\begin{lemm}
Assume that $X$ is negatively curved,
and let $\gamma$ be a maximal geodesic of $X$, $\mu$ be in $\wass(X)$.
Given a point $x\in\gamma$, denote by $(x^t\cdot\mu)_{t\in[0,1]}$ the geodesic segment from 
$\delta_x$ to $\mu$. 

The measure $\mu$ is supported on $\gamma$ if and only if for all
$x,g\in\gamma$
\[\dw(x^{\frac12}\cdot\mu,x^{\frac12}\cdot\delta_g)=\frac12\dw(\mu,\delta_g).\]
\end{lemm}

\begin{proof}
If $x, y$ are points of $X$, $(x+y)/2$ denotes the midpoint of $x$ and $y$. The ``only if'' part is obvious since the transport problem on a convex
subset of $X$ only involves the induced metric on this subset, which here
is isometric to an interval.

To prove the ``if'' part, first notice that the CAT(0) inequality
in a triangle $(x,y,g)$ yields the Thales inequality
$d((x+y)/2,(x+g)/2)\leqslant\frac12 d(y,g)$, so that by direct integration
\begin{equation}
\dw(x^{\frac12}\cdot\mu,\delta_{(x+g)/2})\leqslant\frac12\dw(\mu,\delta_g).
\label{eq:thales}
\end{equation}
Note also that $x^{\frac12}\cdot\delta_g=\delta_{(x+g)/2}$.

Assume now that $\mu$ is not supported on $\gamma$ 
and let $y\in\supp\mu\setminus\gamma$, and $x,g\in\gamma$
such that $x,y,g$ are not aligned. Since $\gamma$ is maximal, this is possible, for example by taking
$d(x,g)>d(x,y)$ in the branching case (any choice of $x\neq g$ would do
in the non-branching case). Since $X$ is negatively curved, we get
that $d((x+y)/2,(x+g)/2)<\frac12 d(y,g)$ so that
\eqref{eq:thales} is strict (see Figure \ref{fig:geodsupported}).
\end{proof}

\begin{figure}[htp]\begin{center}
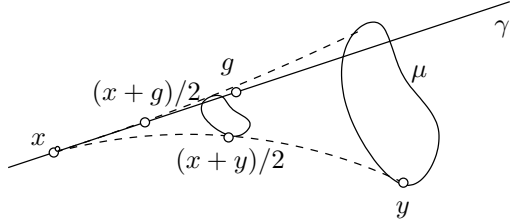
\caption{In negative curvature,
  midpoints are closer than they would be in Euclidean space.}\label{fig:geodsupported}
\end{center}\end{figure}

\begin{coro}
If $X$ is a negatively curved Hadamard space and $\gamma$ is
a maximal geodesic of $X$, any isometry
of $\wass(X)$ that fixes Dirac masses must
preserve the subset $\wass(\gamma)$ of measures supported on $\gamma$,
and therefore induces an isometry on this set.
\end{coro}

\begin{proof}
Let $\varphi$ be an isometry of $\wass(X)$ that fixes Dirac masses.
Let $\gamma$ be a maximal geodesic of $X$ and $\mu\in \wass(X)$ be 
supported on $\gamma$. If $\varphi(\mu)$ were not supported in $\gamma$,
there would exist $x,g\in\gamma$ such that 
$\dw(x^{\frac12}\cdot\varphi(\mu),\delta_{(x+g)/2})<\frac12\dw(\varphi(\mu),\delta_g)$.
But $\varphi$ is an isometry and $\varphi^{-1}(\delta_x)=\delta_x$,
$\varphi^{-1}(\delta_g)=\delta_g$, $\varphi^{-1}(\delta_{(x+g)/2})=\delta_{(x+g)/2}$
so that also $x^{\frac12}\cdot\varphi(\mu)=\varphi(x^{\frac12}\cdot\mu)$.
As a consequence, \eqref{eq:thales} would be a strict inequality too, a contradiction.
\end{proof}

\subsection{Isometry induced on a geodesic}

We want to deduce from the previous section that an isometry that fixes all Dirac masses must also fix every
geodesically-supported measure. To this end, we have to show that the isometry
induced on the measures supported in a given geodesic is trivial, in other terms to
rule out the other possibilities exhibited in \cite[Section 5]{Kloeckner}, at which it is recommended
to take a look before reading the proof below.

\begin{prop}\label{prop:geod_supported}
Assume that $X$ is a geodesically complete, negatively curved Hadamard space and let $\varphi$
be an isometry of $\wass(X)$ that fixes all Dirac masses. For all complete geodesics
$\gamma$ of $X$, the isometry induced by $\varphi$
on $\wass(\gamma)$ is the identity.
\end{prop}

We only address the geodesically complete case for simplicity, but the same result
probably  holds in more generality.

\begin{proof}
Let $x$ be a point not lying on $\gamma$. Such a point exists since $X$ is negatively curved,
hence not a line.

First assume that $\varphi$ induces on $\wass(\gamma)$ an exotic isometry.
Let $y,z$ be two points of $\gamma$ and define $\mu_0=\frac12\delta_y+\frac12\delta_z$.
Then $\mu_n:=\varphi^n(\mu_0)$ has the form $m_n\delta_{y_n}+(1-m_n)\delta_{z_n}$
where $m_n\to 0$, $y_n\to\infty$, $z_n\to(y+z)/2=:g$.

Let now $\gamma'$ be a complete geodesic that contains $y'=(x+y)/2$ and $z'=(x+z)/2$
(see Figure \ref{fig:exotic}). Since
the midpoint of $\mu_0$ and $\delta_x$ is supported on $\gamma'$, so is the midpoint
of $\mu_n$ and $\delta_x$. This means that $(x+y_n)/2$ and $(x+z_n)/2$ lie on
$\gamma'$, and the former goes to infinity. This shows that $\varphi$ also
induces an exotic isometry on $\wass(\gamma')$, thus that $(x+z_n)/2$ goes to
the midpoint $g'$ of $y'$ and $z'$. But we already know that $(x+z_n)/2$ tends
to the midpoint of $x$ and $g$, which must therefore be $g'$. 

\begin{figure}[htp]\begin{center}
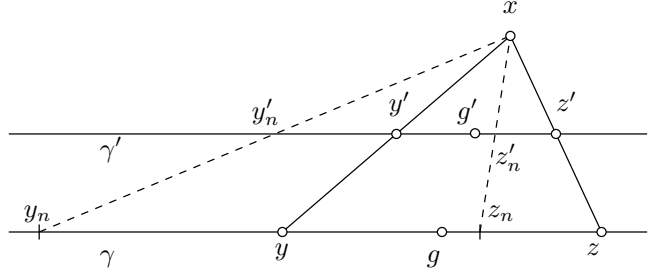
\caption{That an isometry acts exotically on a geodesic $\gamma$ would imply
  that there is another geodesic $\gamma'$ such that taking midpoints with $x$ is 
  an affine map. Then $\gamma$ and $\gamma'$ would be parallel, therefore they would
  bound a flat strip.}
  \label{fig:exotic}
\end{center}\end{figure}

Since this holds
for all choices of $y$ and $z$, we see that the map $y\to(x+y)/2$ maps
affinely $\gamma$ to $\gamma'$. But the geodesic segment $[x\gamma_t]$ converges
when $t\to\pm\infty$ to geodesic segments asymptotic to $\gamma$ and $-\gamma$ respectively.
It follows that $\gamma'$ is parallel to $\gamma$, so that they must bound a flat strip
(see \cite{Ballmann}). But this is forbidden by the negative curvature assumption.

A similar argument can be worked out in the case when $\varphi$ induces an involution:
given $y,z\in\gamma$ and their midpoint $g$,
one can find measures $\mu_n$ supported on $y$ and $y_n$, where $y_n\to g$ and $\mu_n$ has more mass on
$y_n$ than on $y$, such that $\varphi(\mu_n)$
is supported on $z$ and a point $z_n$ of $\gamma$, with more mass on $z_n$. It follows
that the midpoints $y',y'_n,z'$ with $x$ are on a line, and we get the same contradiction
as before.

The classification of isometries of $\wass(\mathbb{R})$ shows that if $\varphi$ fixes
Dirac masses and is neither an exotic isometry nor an involution, then it is the identity.
\end{proof}

Now we are able to link the isometric rigidity of $\wass(X)$ to the injectivity of a Radon
transform. The following definition relies on the following observation: since a geodesic is convex
and $X$ is Hadamard, given a point $y$ and a geodesic $\gamma$ there is 
a unique point $p_\gamma(y)\in\gamma$ closest to $y$, called the projection of $y$ to $\gamma$.

\begin{defi}
When $X$ is geodesically complete,
we define the \emph{perpendicular Radon transform} $\radon\mu$
of a measure $\mu\in\wass(X)$ as the following map defined over 
complete geodesics $\gamma$ of $X$:
\[\radon\mu(\gamma)=(p_\gamma)_\#\mu.\]
In other words, this Radon transform recalls all the projections of a measure
on geodesics.
\end{defi}

The following result is now a direct consequence of Proposition \ref{prop:geod_supported}.
\begin{prop}\label{prop:inversion}
Assume that $X$ is geodesically complete and negatively curved.
If there is a dense subset $A\subset \wass(X)$ such that
for all $\mu\in\wass(X)$ and all $\nu\in A$, we have
\[\radon\mu=\radon\nu \Rightarrow \mu=\nu\]
then $\wass(X)$ is isometrically rigid.
\end{prop}

\begin{proof}
Let $\varphi$ be an isometry of $\wass(X)$. Up to composing with an element
of the image of $\#$, we can assume that $\varphi$ acts trivially on Dirac masses.
Then it acts trivially on geodesically supported measures. For all $\nu\in A$,
since $(p_\gamma)_\#\nu$ is the measure supported on $\gamma$ closest to $\nu$,
one has $\radon\varphi(\nu)=\radon\nu$, hence $\varphi(\nu)=\nu$. We just proved
that $\varphi$ acts trivially on a dense set, so that it must be the identity.
\end{proof}

\section{Injectivity of the Radon transform}

The proof of Theorem \ref{theo:intro:isometry} shall be complete as soon
as we get the injectivity required in Proposition \ref{prop:inversion}.
Note that in the case of the real
hyperbolic space $\RH^n$, we could use the usual Radon transform on the set of
compactly supported measures with smooth density to get it (see \cite{Helgason}).

\subsection{The case of manifolds and their siblings}\label{blurb}

Let us first give an argument that does the job for all manifolds. We shall give a more
general, but somewhat more involved argument afterwards.
\begin{prop}
Assume that $X$ is a Hadamard smooth manifold . Let $A$ be the set of finitely supported
measures. For all $\mu\in\wass(X)$ and all $\nu\in A$, if $\radon\mu=\radon\nu$
then $\mu=\nu$.
\end{prop}

Note that $A$ is dense in $\wass(X)$ so that this proposition ends the proof of Theorem
\ref{theo:intro:isometry} in the case of manifolds.

\begin{proof}
Write $\nu=\sum m_i\delta_{x_i}$ where $\sum m_i=1$. Note that since $X$ is a
\emph{negatively} curved manifold, it has dimension at least $2$.

First, we prove under the assumption $\radon\mu=\radon\nu$ that $\mu$ must be supported
on the $x_i$. Let $x$ be any other point, and consider a geodesic $\gamma$ such that
$\gamma_0=x$ and $\dot\gamma_0$ is not orthogonal to any of the geodesics $(xx_i)$.
Then for all $i$, $p_\gamma(x_i)\neq x$ and there is an $\varepsilon>0$ such that
the neighborhood of size $\varepsilon$ around $x$ on $\gamma$ does not contain any
of these projections. It follows that $\radon\nu(\gamma)$ is supported outside
this neighborhood, and so does $\radon\mu(\gamma)$. But the projection on $\gamma$ is $1$-Lipschitz,
so that $\mu$ must be supported outside the $\varepsilon$ neighborhood of $x$ in $X$. In particular,
$x\notin\supp\mu$.

Now, if $\gamma$ is a geodesic containing $x_i$, then $\radon\nu(\gamma)$ is finitely supported
with a mass at least $m_i$ at $x_i$. For a generic $\gamma$, the mass at $x_i$ is exactly
$m_i$. It follows immediately, since $\mu$ is supported on the $x_i$, that its mass at $x_i$
is $m_i$.
\end{proof}

The above proof mainly uses the fact that given a point $x$ and a finite number
of other points $x_i$, there is a geodesic $\gamma\ni x$ such that $p_\gamma(x_i)\neq x$ for
all $i$. It follows that the proof can be adjusted to get the general case.
To this end, we need to introduce some extra definitions.


\subsection{Regular subset of a Hadamard space}

Assume that $X$ is a Hadamard geodesically complete space and let $p \in X$ be a point. 
We set $\Sigma_p'$ the set 
of all nontrivial geodesics starting at $p$. The angle $\angle$ is a pseudo-metric on $\Sigma_p'$. The 
space of directions $\Sigma_p$ at $p$ is the completion with respect to $\angle$ of the quotient metric 
space obtained from $\Sigma_p'$ by the relation $\angle =0$. Under these assumptions, the space of 
directions at $p$ is a compact $\CAT(1)$ space whose diameter is smaller or equal to $\pi$ (see 
\cite{BBI} for a proof). We shall also use the geometric dimension of a CAT space introduced by Kleiner 
in \cite{kleiner}.

\begin{defi}[Geometric dimension] Let $U$ be an open subset of a locally compact CAT(1) space. Then the dimension of $U$ is $0$ if $U$ is discrete. Otherwise, it is defined as
$$ \dim U= 1 + \sup_{p \in U} \dim \Sigma_p.$$
\end{defi}

Given a  geodesically complete Hadamard space $X$, one defines its {\it regular set} $\Reg(X)$ as the subset of $X$ made of points $p$ whose space of directions $\Sigma_p$ is isometric to a standard sphere $\S^k$ where $k$ is any integer (or equivalently, whose tangent cone at $p$ is isometric to the Euclidean space of dimension $k+1$). When $k=0$, we further require the existence of an open neighborhood of $p$ which is isometric to an open segment in $\R$. 

\begin{prop}
Assume that $X$ is a  geodesically complete Hada\-mard space.
Let $A$ be the set of measures supported on finite set of points all
located in the regular set of $X$.
For all $\mu\in\wass(X)$ and all $\nu\in A$, 
if $\radon\mu=\radon\nu$ then $\mu=\nu$.
\end{prop}

\begin{proof}
Write $\nu=\sum m_i\delta_{x_i}$ where $\sum m_i=1$.
Consider one of the $x_i$, and let $m'_i:=\mu(\{x_i\})$. Let us first assume that $\Sigma_{x_i}$ is isometric to $\S^{n_i}$ with $n_i\geq 1$.

Let $\gamma$ be a geodesic and $x$ be a point that does not belong to the image of $\gamma$. Then, the angle at $y=p_{\gamma}(x)$ between the geodesic $\sigma$ from $y$ to $x$ and $\gamma$ satisfies $\angle  _y (\sigma', \gamma') \geq \pi/2$ according to the first variation formula. The same argument yields $\angle  _y (\sigma', -\gamma') \geq \pi/2$. Thus, if we further assume that $\Sigma_y= \S^k$ with $k\geq 1$, we  finally get
$$  \angle_y (\sigma', \gamma') =\angle _y (\sigma', -\gamma') = \pi/2.  $$

Using this property, it is now easy to find a geodesic $\gamma$ with $x_i \in \gamma$ and $p_{\gamma}(x_j) \neq x_i$ for all $j \neq i$. From this and $\radon\mu=\radon\nu$, we get $m'_i\leq m_i$. It also follows
that for some $\varepsilon>0$, the measure $\radon\mu(\gamma)=\radon\nu(\gamma)$
is concentrated outside $B(x_i,\varepsilon)\setminus{x_i}$.
Since $p_{\gamma}$ is one Lipschitz (see \cite[Proposition 2.4]{BH}), it also follows that $\mu$ is concentrated outside 
$B(x_i,\varepsilon)\setminus{p_\gamma^{-1}(x_i)}$; in particular $\mu$
is concentrated outside $B(x_i,\varepsilon)\setminus{x_i}$.

Let $\gamma^\perp$ denote the set of points $x$ such that the geodesic segment
$[xx_i]$ is orthogonal to $\gamma$ at $x_i$. By definition of the Radon transform, $\mu(\gamma^\perp\setminus x_i)=m_i-m'_i$ and
for all $x\in\supp\mu\setminus \gamma^\perp$, we have
$p_\gamma(x)\notin B(x_i,\varepsilon)$.

Choose a second geodesic $\gamma_2\ni x_i$, close enough to $\gamma$ to ensure
that for all $x\in\supp\mu\setminus \gamma^\perp$, we still have
$p_{\gamma_2}(x)\notin B(x_i,\varepsilon)$ (up to shrinking $\varepsilon$ a bit if necessary).
We can moreover assume that $\gamma_2$
enjoys the same properties we asked to $\gamma$, so that
\[\forall x\in\supp\mu\setminus \gamma^\perp\cap\gamma_2^\perp,\quad 
p_{\gamma_2}(x)\notin B(x_i,\varepsilon).\]

Recall that $\dim \Sigma_{x_i}= n_i\geq 1$. We can construct inductively
a family $\gamma_1,\ldots,\gamma_{n_i}$
of geodesics chosen as above, and such that their
velocity vectors at $x_i$ span its tangent space $T_{x_i}X$ which is isometric to $\R^{n_i}$. For all points $x$ in 
\[\supp\mu\setminus \cap_\alpha\gamma_\alpha^\perp,\]
we get $p_{\gamma_n}(x)\notin B(x_i,\varepsilon)$. But $\cap_\alpha\gamma_\alpha^\perp=\{x_i\}$,
and considering $\radon\mu(\gamma_n)=\radon\nu(\gamma_n)$ we get $m'_i=m_i$. 

It remains to treat the case when $\Sigma_{x_i}$ is isometric to $\S^0$. Recall that in that case, we 
further assume the existence of an open neighborhood $V$ of $x_i$ isometric to a short open segment. It 
is then clear that choosing any geodesic $\gamma$ going through $x_i$, $p_{\gamma}(x)\neq x_i$ for any 
$x\neq x_i$. 

Since
$\mu$ is a probability measure and $\sum m_i=1$, we deduce $\mu=\nu$.  
\end{proof}

To get our main result, it remains to prove that $\Reg(X)$ is a dense subset of $X$. This is a 
consequence of deep results of Lytchak and Nagano \cite{LN} on the structure of spaces with an upper 
curvature bound. 


\subsection{Fine properties of Hadamard spaces}

In this section, we report the results we need from the paper \cite{LN}, some of these results are also 
contained in \cite{OT}. We then use them below to complete the proof of Theorem 
\ref{theo:intro:isometry} in 
the general case. For simplicity, we give the statements for a geodesically complete Hadamard space $X$ 
only, see \cite{LN} for the full results.

Given $U$ a relatively compact open subset of $X$, the authors introduce the set of \emph{$n$-regular}
points $R_n(U)$ defined by 

$$ R_n(U)=\{x \in U; \Sigma_x \mbox{ is isometric  to } \S^{n-1}*Z\}$$
where $(Z,d)$ is a metric space and $\S^{n-1}*Z$ is the spherical join of $\S^{n-1}$ and $Z$.
We shall use the following results they prove:
\begin{theo}[Lytchak-Nagano]\label{LN1} Let $U$ be a locally compact open subset of $X$. 
Then, the Hausdorff dimension of $U\setminus R_n(U)$ satisfies
$$\dim_{\Ha} (U\setminus R_n(U)) \leq n-1.$$
\end{theo}
(This is \cite[Theorem 1.1]{LN}.)

\begin{lemm}[Lytchak-Nagano] \label{LN2}
Let $U$ be a relatively compact 
open subset of $X$ and assume that $\dim U=1$. Then, for each $x \in R_1(U)$,
there exists a bilipschitz embedding of a small open neighborhood of $x$ into $\R$.
\end{lemm}
(Obtained by combining \cite[Lemma 9.6]{LN} (see also the discussion above the statement) and 
\cite[Lemma 11.5]{LN}.)
 
\begin{theo}[Otsu-Tanoue, Lytchak-Nagano]\label{LN3} 
Let $U$  be a relatively compact open subset of $X$. Then, there exists an integer $n \in \N$ such that 
$\Ha^n(U) \in (0,+\infty)$. Moreover, the geometric dimension of $U$ coincides with its Hausdorff 
dimension, namely $\dim U= n$.
\end{theo}
(This is \cite[Theorem 13.1]{LN} which is a refined version of earlier result in \cite{OT}.)
  
\begin{rema} At first sight, it is not even clear that the geometric dimension $\dim U$ is finite. To 
prove this, first note that
$$ \dim (U) \leq \dim_{\Ha} (U)$$
thanks to a result due to Kleiner \cite[Theorem A]{kleiner} and the fact that the Hausdorff dimension is not smaller than the topological dimension. 
Then, we claim that
$$ \dim_{\Ha}(U) \leq \dim_r (U)$$
where $\dim_r$ stands for the rough dimension introduced by Burago-Gromov-Perelman, see \cite[Chapter 10]{BBI}. Thus, it suffices to convince
oneself that 
\begin{equation}\label{rdim}
\dim_r (U)<+\infty.
\end{equation}
 To this aim, Lytchak and Nagano prove that $(U,d)$ is a doubling metric space (contrary to what precedes, this 
is strongly based on the geodesic completeness assumption). This property easily implies (\ref{rdim}).
\end{rema}  

\subsection{Density of the regular set}

In this section, we use the tools described above to prove the following statement,
completing the proof of Theorem \ref{theo:intro:isometry}.

\begin{theo}Let $X$ be a geodesically complete Hadamard space. 
Then $\Reg(X)$ is a dense subset of $X$.
\end{theo}
 
\begin{proof}
Given any $x \in X$, we have to prove that $x$ is in the closure of $\Reg(X)$.

We first introduce a definition: we call the minimal dimension of a neighborhood of 
$x$ the \emph{local dimension} at $x$ and denote it by $n_x$; moreover
there is an open, relatively compact neighborhood $U$ of $x$ such that $\dim U=n_x$.
To see this, recall that $X$ must have relatively compact balls, and
apply Theorem \ref{LN3} to $(\dim B(x,1/k))_{k\ge 1}$: it
is a decreasing sequence with positive integer value, thus has a limit which is reached for some
$k$. Moreover, we get that $n_x$ is also the Hausdorff dimension of $U$.

The theorem follows from two claims: first $R_{n_x}(U) \subset \Reg(X)$,
second $x$ is in the closure of $R_{n_x}(U)$.

To prove the first claim, observe that if $Z$ is any non-empty space, the compactness of $\Sigma_x$ then yields that $\mathbb{S}^{n_x-1} * Z$
has dimension at least $n_x$; since $\dim U = n_x$, this shows that all points of 
$R_{n_x}(U)$ have their space of direction isometric to $\mathbb{S}^{n_x-1}$.
When $n_x=1$, Lemma \ref{LN2} further ensures that points in $R_{n_x}(U)$
have a neighborhood isometric to an interval, proving the claim.

We prove the second claim by contradiction. If $x$ were not in the closure of $R_{n_x}(U)$, 
by Theorem \ref{LN1} it would lie in an open set $U'\subset U$ of Hausdorff
dimension less than $n_x$. By Theorem \ref{LN3},  $U'$ has dimension less
than $n_x$, contradicting the definition of local dimension.
\end{proof}

\section{Appendix: the Radon transform on trees}

In this appendix we prove a side result, not needed in the proof of Theorem
\ref{theo:intro:isometry}, showing that in the case of simplicial trees one can
explicitly inverse the Radon transform introduced above.

Let $X$ be a locally finite tree that is not a line. 
We describe $X$ by a couple
$(V,E)$ where $V$ is the set of vertices;
$E$ is the set of edges, each endowed with one or two endpoints
in $V$ and a length. 

For all $x\in V$, let $k(x)$ be the valency of $x$, that is
the number of edges incident to $x$.
We assume that no vertex has valency $2$, since
otherwise we could describe the same metric space by a simpler graph.
Note also that edges with only one endpoint have infinite length since
$X$ is assumed to be complete.

In this setting, $X$ is geodesically complete if and only if 
it has no leaf (vertex of valency $1$).
Assume this, and let $\gamma$ be a complete geodesic.
If $x$ is a vertex lying on $\gamma$ and
$C_1,\ldots,C_{k(x)}$ are the connected components of $X\setminus\{x\}$,
let $\perpend^x(\gamma)$ be the union of $x$ and of the $C_i$ not meeting $\gamma$
(see Figure \ref{fig:perpendicular}).
It inherits a tree structure from $X$. In fact, $\perpend^x(\gamma)$
depends only upon the two edges $e,f$ of $\gamma$ that are incident to
$x$. We therefore let $\perpend^x(ef)=\perpend^x(\gamma)$.

\begin{figure}[htp]\begin{center}
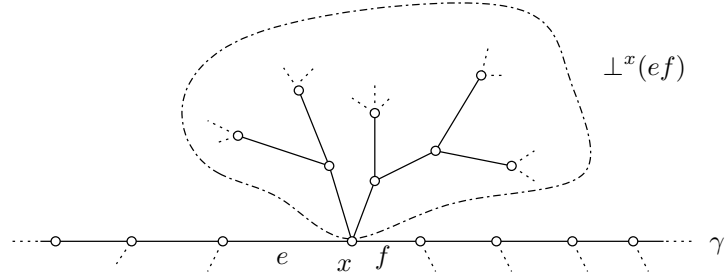
\caption{A perpendicular to a geodesic.}\label{fig:perpendicular}
\end{center}\end{figure}

The levels of $p_\gamma$ are called the \emph{perpendiculars} of $\gamma$,
they also are the bissectors of its points. They are exactly:
\begin{itemize}
\item the sets $\perpend^x(\gamma)$ where $x$ is a vertex of $\gamma$.
\item the sets $\{x\}$ where $x$ is a point interior to an edge of $\gamma$.
\end{itemize}

A measure $\mu\in\wass(X)$ can be decomposed into a part supported outside vertices,
which is obviously determined by the projections of $\mu$ on the various geodesics
of $X$, and an atomic part supported on vertices. Therefore, we are reduced to
study the perpendicular Radon transform reformulated as follows for functions
defined on $V$ instead of measures on $X$.

\begin{defi}[combinatorial Radon transform]
A \emph{flag} of $X$ is defined as a triple $(x,ef)$ where
$x$ is a vertex, $e \neq f$ are edges incident to $x$ and $ef$
denotes an unordered pair.
Let us denote the set of flags by $F$; we write $x\in ef$
to say that $(x,ef)\in F$.

Given  a summable function $h$ defined on the vertices of $X$, we
define its \emph{combinatorial perpendicular Radon transform} as the map
\begin{eqnarray*}
\radon h :\qquad F &\to& \mathbb{R} \\
           (x,ef) &\mapsto& \sum_{y\in \perpend^x(ef)} h(y)
\end{eqnarray*}
where the sum is on vertices of $\perpend^x(ef)$.
\end{defi}

It seems that this Radon transform has not been studied before,
contrary to the transforms defined using geodesics \cite{Berenstein},
horocycles \cite{Cartier, Betori} and circles \cite{Casadio}.

\begin{theo}[Inversion formula]
Two maps $h,l:V\to\mathbb{R}$ such that $\sum h=\sum l$ and
$\radon h=\radon l$ are equal.
More precisely, we can recover $h$ from $\radon h$ by the following
inversion formula:
\[h(x) = \frac{1}{k(x)-1}\sum_{ef\ni x}\radon h(x,ef)-\frac{k(x)-2}2 \sum_{y\in V} h(y)\]
where the first sum is over the set of pairs of edges incident to $x$.
\end{theo}

\begin{proof}
The formula relies on a simple double counting argument:
\begin{eqnarray*}
\sum_{ef\ni x}\radon h(x,ef) &=& \sum_{ef\ni x} \,\sum_{y\in\perpend^x(ef)} h(y)\\
  &=& \sum_{y\in V} h(y) n_x(y)
\end{eqnarray*}
where $n_x(y)$ is the number of flags $(x,ef)$ such that $y\in\perpend^x(ef)$.
If $y\neq x$, let $e_y$ be the edge incident to $x$ starting
the geodesic segment from $x$ to $y$. Then $y\in\perpend^x(ef)$ if and only
if $e,f\neq e_y$. Therefore, $n_x(y)=\binom{k(x)-1}{2}$. But $n_x(x)=\binom{k(x)}{2}$,
so that
\[\sum_{ef\ni x}\radon h(x,ef)=\binom{k(x)-1}{2}\sum_{y\in V} h(y)+ (k(x)-1)h(x).\]
\end{proof}

\begin{acknowledgements} J. Bertrand would like to thank Sasha Lytchak for useful discussions and the talk he gave in Toulouse on his joint work \cite{LN} with K. Nagano.
\end{acknowledgements}

\bibliographystyle{alpha}
\bibliography{biblio}

\end{document}

%% file: cat.pdf_tex
\begin{picture}(0,0)%
\includegraphics{cat.pdf}%
\end{picture}%
\setlength{\unitlength}{4144sp}%
\begingroup\makeatletter\ifx\SetFigFont\undefined%
\gdef\SetFigFont#1#2#3#4#5{%
  \reset@font\fontsize{#1}{#2pt}%
  \fontfamily{#3}\fontseries{#4}\fontshape{#5}%
  \selectfont}%
\fi\endgroup%
\begin{picture}(3592,1357)(3631,-3525)
\put(3646,-2311){\makebox(0,0)[lb]{\smash{{\SetFigFont{10}{12.0}{\rmdefault}{\mddefault}{\updefault}{\color[rgb]{0,0,0}$x$}%
}}}}
\put(4456,-3391){\makebox(0,0)[lb]{\smash{{\SetFigFont{10}{12.0}{\rmdefault}{\mddefault}{\updefault}{\color[rgb]{0,0,0}$y$}%
}}}}
\put(5176,-2311){\makebox(0,0)[lb]{\smash{{\SetFigFont{10}{12.0}{\rmdefault}{\mddefault}{\updefault}{\color[rgb]{0,0,0}$z$}%
}}}}
\put(4141,-2446){\makebox(0,0)[lb]{\smash{{\SetFigFont{10}{12.0}{\rmdefault}{\mddefault}{\updefault}{\color[rgb]{0,0,0}$\gamma(t)$}%
}}}}
\put(5765,-2291){\makebox(0,0)[lb]{\smash{{\SetFigFont{10}{12.0}{\rmdefault}{\mddefault}{\updefault}{\color[rgb]{0,0,0}$x'$}%
}}}}
\put(6446,-3465){\makebox(0,0)[lb]{\smash{{\SetFigFont{10}{12.0}{\rmdefault}{\mddefault}{\updefault}{\color[rgb]{0,0,0}$y'$}%
}}}}
\put(7098,-2326){\makebox(0,0)[lb]{\smash{{\SetFigFont{10}{12.0}{\rmdefault}{\mddefault}{\updefault}{\color[rgb]{0,0,0}$z'$}%
}}}}
\end{picture}%

%% file: aligned.pdf_tex
\begin{picture}(0,0)%
\includegraphics{aligned.pdf}%
\end{picture}%
\setlength{\unitlength}{4144sp}%
\begingroup\makeatletter\ifx\SetFigFontNFSS\undefined%
\gdef\SetFigFontNFSS#1#2#3#4#5{%
  \reset@font\fontsize{#1}{#2pt}%
  \fontfamily{#3}\fontseries{#4}\fontshape{#5}%
  \selectfont}%
\fi\endgroup%
\begin{picture}(1559,1602)(5390,-2686)
\put(5491,-1231){\makebox(0,0)[lb]{\smash{{\SetFigFontNFSS{10}{12.0}{\rmdefault}{\mddefault}{\updefault}{\color[rgb]{0,0,0}$\mu_0$}%
}}}}
\put(6931,-2626){\makebox(0,0)[lb]{\smash{{\SetFigFontNFSS{10}{12.0}{\rmdefault}{\mddefault}{\updefault}{\color[rgb]{0,0,0}$y''$}%
}}}}
\put(5761,-1411){\makebox(0,0)[lb]{\smash{{\SetFigFontNFSS{10}{12.0}{\rmdefault}{\mddefault}{\updefault}{\color[rgb]{0,0,0}$y'$}%
}}}}
\put(6391,-2446){\makebox(0,0)[lb]{\smash{{\SetFigFontNFSS{10}{12.0}{\rmdefault}{\mddefault}{\updefault}{\color[rgb]{0,0,0}$y$}%
}}}}
\end{picture}%

%% file: geodsupported.pdf_tex
\begin{picture}(0,0)%
\includegraphics{geodsupported.pdf}%
\end{picture}%
\setlength{\unitlength}{4144sp}%
\begingroup\makeatletter\ifx\SetFigFont\undefined%
\gdef\SetFigFont#1#2#3#4#5{%
  \reset@font\fontsize{#1}{#2pt}%
  \fontfamily{#3}\fontseries{#4}\fontshape{#5}%
  \selectfont}%
\fi\endgroup%
\begin{picture}(2994,1332)(4489,-4171)
\put(4636,-3706){\makebox(0,0)[lb]{\smash{{\SetFigFont{10}{12.0}{\rmdefault}{\mddefault}{\updefault}{\color[rgb]{0,0,0}$x$}%
}}}}
\put(7381,-3031){\makebox(0,0)[lb]{\smash{{\SetFigFont{10}{12.0}{\rmdefault}{\mddefault}{\updefault}{\color[rgb]{0,0,0}$\gamma$}%
}}}}
\put(6796,-4111){\makebox(0,0)[lb]{\smash{{\SetFigFont{10}{12.0}{\rmdefault}{\mddefault}{\updefault}{\color[rgb]{0,0,0}$y$}%
}}}}
\put(4996,-3436){\makebox(0,0)[lb]{\smash{{\SetFigFont{10}{12.0}{\rmdefault}{\mddefault}{\updefault}{\color[rgb]{0,0,0}$(x+g)/2$}%
}}}}
\put(5491,-3841){\makebox(0,0)[lb]{\smash{{\SetFigFont{10}{12.0}{\rmdefault}{\mddefault}{\updefault}{\color[rgb]{0,0,0}$(x+y)/2$}%
}}}}
\put(6886,-3301){\makebox(0,0)[lb]{\smash{{\SetFigFont{10}{12.0}{\rmdefault}{\mddefault}{\updefault}{\color[rgb]{0,0,0}$\mu$}%
}}}}
\put(5761,-3256){\makebox(0,0)[lb]{\smash{{\SetFigFont{10}{12.0}{\rmdefault}{\mddefault}{\updefault}{\color[rgb]{0,0,0}$g$}%
}}}}
\end{picture}%

%% file: exotic.pdf_tex
\begin{picture}(0,0)%
\includegraphics{exotic.pdf}%
\end{picture}%
\setlength{\unitlength}{4144sp}%
\begingroup\makeatletter\ifx\SetFigFont\undefined%
\gdef\SetFigFont#1#2#3#4#5{%
  \reset@font\fontsize{#1}{#2pt}%
  \fontfamily{#3}\fontseries{#4}\fontshape{#5}%
  \selectfont}%
\fi\endgroup%
\begin{picture}(3804,1668)(3319,-4981)
\put(4906,-4876){\makebox(0,0)[lb]{\smash{{\SetFigFont{10}{12.0}{\rmdefault}{\mddefault}{\updefault}{\color[rgb]{0,0,0}$y$}%
}}}}
\put(6751,-4876){\makebox(0,0)[lb]{\smash{{\SetFigFont{10}{12.0}{\rmdefault}{\mddefault}{\updefault}{\color[rgb]{0,0,0}$z$}%
}}}}
\put(3421,-4651){\makebox(0,0)[lb]{\smash{{\SetFigFont{10}{12.0}{\rmdefault}{\mddefault}{\updefault}{\color[rgb]{0,0,0}$y_n$}%
}}}}
\put(5581,-4021){\makebox(0,0)[lb]{\smash{{\SetFigFont{10}{12.0}{\rmdefault}{\mddefault}{\updefault}{\color[rgb]{0,0,0}$y'$}%
}}}}
\put(6571,-4021){\makebox(0,0)[lb]{\smash{{\SetFigFont{10}{12.0}{\rmdefault}{\mddefault}{\updefault}{\color[rgb]{0,0,0}$z'$}%
}}}}
\put(6256,-3436){\makebox(0,0)[lb]{\smash{{\SetFigFont{10}{12.0}{\rmdefault}{\mddefault}{\updefault}{\color[rgb]{0,0,0}$x$}%
}}}}
\put(3871,-4921){\makebox(0,0)[lb]{\smash{{\SetFigFont{10}{12.0}{\rmdefault}{\mddefault}{\updefault}{\color[rgb]{0,0,0}$\gamma$}%
}}}}
\put(3871,-4291){\makebox(0,0)[lb]{\smash{{\SetFigFont{10}{12.0}{\rmdefault}{\mddefault}{\updefault}{\color[rgb]{0,0,0}$\gamma'$}%
}}}}
\put(5986,-4066){\makebox(0,0)[lb]{\smash{{\SetFigFont{10}{12.0}{\rmdefault}{\mddefault}{\updefault}{\color[rgb]{0,0,0}$g'$}%
}}}}
\put(6211,-4336){\makebox(0,0)[lb]{\smash{{\SetFigFont{10}{12.0}{\rmdefault}{\mddefault}{\updefault}{\color[rgb]{0,0,0}$z'_n$}%
}}}}
\put(4771,-4066){\makebox(0,0)[lb]{\smash{{\SetFigFont{10}{12.0}{\rmdefault}{\mddefault}{\updefault}{\color[rgb]{0,0,0}$y'_n$}%
}}}}
\put(5806,-4921){\makebox(0,0)[lb]{\smash{{\SetFigFont{10}{12.0}{\rmdefault}{\mddefault}{\updefault}{\color[rgb]{0,0,0}$g$}%
}}}}
\put(6166,-4651){\makebox(0,0)[lb]{\smash{{\SetFigFont{10}{12.0}{\rmdefault}{\mddefault}{\updefault}{\color[rgb]{0,0,0}$z_n$}%
}}}}
\end{picture}%

%% file: perpendicular.pdf_tex
\begin{picture}(0,0)%
\includegraphics{perpendicular.pdf}%
\end{picture}%
\setlength{\unitlength}{4144sp}%
\begingroup\makeatletter\ifx\SetFigFontNFSS\undefined%
\gdef\SetFigFontNFSS#1#2#3#4#5{%
  \reset@font\fontsize{#1}{#2pt}%
  \fontfamily{#3}\fontseries{#4}\fontshape{#5}%
  \selectfont}%
\fi\endgroup%
\begin{picture}(4167,1669)(1969,-4076)
\put(6121,-3886){\makebox(0,0)[lb]{\smash{{\SetFigFontNFSS{10}{12.0}{\rmdefault}{\mddefault}{\updefault}{\color[rgb]{0,0,0}$\gamma$}%
}}}}
\put(5491,-2851){\makebox(0,0)[lb]{\smash{{\SetFigFontNFSS{10}{12.0}{\rmdefault}{\mddefault}{\updefault}{\color[rgb]{0,0,0}$\perpend^x(ef)$}%
}}}}
\put(4141,-3976){\makebox(0,0)[lb]{\smash{{\SetFigFontNFSS{10}{12.0}{\rmdefault}{\mddefault}{\updefault}{\color[rgb]{0,0,0}$f$}%
}}}}
\put(3916,-4021){\makebox(0,0)[lb]{\smash{{\SetFigFontNFSS{10}{12.0}{\rmdefault}{\mddefault}{\updefault}{\color[rgb]{0,0,0}$x$}%
}}}}
\put(3556,-3976){\makebox(0,0)[lb]{\smash{{\SetFigFontNFSS{10}{12.0}{\rmdefault}{\mddefault}{\updefault}{\color[rgb]{0,0,0}$e$}%
}}}}
\end{picture}%